\documentclass[12pt]{amsart}
\usepackage{mathrsfs, amssymb,amsthm, amsfonts, amsmath,extarrows}
\usepackage[all]{xy}
\usepackage{enumerate}
\numberwithin{equation}{section}

\newtheorem{theorem}[subsection]{Theorem}
\newtheorem*{theorem*}{Theorem}
\newtheorem{lemma}[subsection]{Lemma}
\newtheorem{corollary}[subsection]{Corollary}

\newtheorem{proposition}[subsection]{Proposition}

\theoremstyle{definition}
\newtheorem{example}[subsection]{Example}

\newtheorem*{definition*}{Definition}

\theoremstyle{remark}
\newtheorem{remark}[subsection]{Remark}

\makeatletter\@addtoreset{equation}{section}

\makeatother

\usepackage{hyperref}

\def \KK {\mathbb{K}}
\def \LL {\mathbb{L}}
\def \ZZ {\mathbb{Z}}
\def \QQ {\mathbb{Q}}

\newcommand{\PP}{\mathbb{P}}
\newcommand{\G}{\mathbb{G}}
\newcommand{\Gm}{\mathbb{G}_{\mathrm{m}}}
\newcommand{\Ga}{\mathbb{G}_{\mathrm{a}}}

\newcommand{\Br}{\operatorname{Br}}
\newcommand{\Hom}{\operatorname{Hom}}
\newcommand{\End}{\operatorname{End}}
\newcommand{\Spec}{\operatorname{spec}}
\newcommand{\Aut}{\operatorname{Aut}}

\newcommand{\GL}{\operatorname{GL}}

\newcommand{\PGL}{\operatorname{PGL}}

\newcommand{\Char}{\operatorname{char}}
\newcommand{\Gal}{\operatorname{Gal}}
\newcommand{\rar}[1]{\stackrel{#1}{\longrightarrow}}
\newcommand{\Norm}{\operatorname{Norm}}

\def \ge {\geqslant}
\def \le {\leqslant}

\title{Boundedness for finite subgroups of linear algebraic groups}

\author{Constantin Shramov and Vadim Vologodsky}

\address{\emph{Constantin Shramov}
\newline
\textnormal{Steklov Mathematical Institute of RAS,
8 Gubkina street, Moscow 119991, Russia.
}
\newline
\textnormal{National Research University Higher School of Economics, Laboratory of Algebraic Geometry, NRU HSE, 6 Usacheva str., Moscow, 117312, Russia.
}
\newline
\textnormal{\texttt{costya.shramov@gmail.com}}}

\address{\emph{Vadim Vologodsky}
\newline
\textnormal{National Research University Higher School of Economics,  Laboratory of Mirror Symmetry, NRU HSE, 6 Usacheva str., Moscow, 117312, Russia.
}
\newline
\textnormal{\texttt{vologod@gmail.com}}}

\begin{document}

\begin{abstract}
We show the boundedness of finite subgroups in any anisotropic reductive
group over a perfect field that  contains
all roots of~$1$. Also, we provide explicit bounds for orders of finite
subgroups of automorphism groups of Severi--Brauer varieties and quadrics over
such fields.
\end{abstract}

\maketitle
%\tableofcontents

\section{Introduction}

In this paper we study finite subgroups of linear algebraic groups.
We say that a field~$\KK$ \emph{contains all roots of~$1$},
if, for every positive integer~$n$, the polynomial $x^n-1$ splits completely in $\KK[x]$.
An example of such a field is the field of rational functions on an irreducible variety
defined over an algebraically closed field.
Recall that a linear algebraic group is called \emph{anisotropic}
if it does not contain a subgroup isomorphic to the split one-dimensional torus $\Gm$.
Our main result is the following.

\begin{theorem}\label{theorem:main}
Let $\KK$ be a perfect field that contains all roots of~$1$,
and let $G$ be an anisotropic reductive group
over $\KK$. Then there exists a constant $L=L(G)$
such that any finite subgroup of $G(\KK)$ has order at most~$L$.
\end{theorem}

Actually, we will prove a more precise Theorem~\ref{theorem:LAG}
that gives some boundedness result for non-perfect fields and non-reductive
linear algebraic groups as well, and also provides an explicit (multiplicative) bound for the orders
of finite subgroups in terms of the rank of~$G$, the number of its connected components,
and the minimal dimension of a faithful representation of the maximal reductive quotient of
the neutral component of~$G_{\bar{\KK}}$. Moreover, we will see in Corollary~\ref{corollary:LAG}
that there is a bound that depends only on the rank and the number of connected components of~$G$.

A particular case of Theorem~\ref{theorem:main} for the projective orthogonal groups
was proved by
T.\,Bandman and Yu.\,Zarhin in \cite[Corollary 4.11, Theorem 4.14]{BandmanZarhin2015a}.
They applied its rank~$1$ case to analyze fiberwise
birational maps of varieties fibered into rational curves.

Let us say that a group $G$ has \emph{bounded finite subgroups},
if there exists a constant~\mbox{$L=L(G)$} such that,
for any finite subgroup
$\Gamma\subset G$, one has $|\Gamma|\le L$.
If this is not the case, we say that $G$ \emph{has unbounded finite subgroups}.
Thus, Theorem~\ref{theorem:main} claims that every anisotropic reductive
group over a perfect field that  contains
all roots of~$1$ has bounded finite subgroups.

Recall that a \emph{Severi--Brauer} variety is a variety $X$ over a field $\KK$ such that
its scalar extension to the algebraic closure of $\KK$ is isomorphic to a projective space.
For instance, one-dimensional Severi--Brauer varieties are conics over $\KK$.
The automorphism group scheme of an $(n-1)$-dimensional Severi--Brauer variety is an inner
form of the algebraic group~$\PGL_n$.
One can apply  Theorem~\ref{theorem:main} to study the automorphism groups of Severi--Brauer varieties.
For a Severi--Brauer variety~$X$ associated to a central simple algebra $A$ over a perfect  field $\KK$
that contains all roots of~$1$, Theorem~\ref{theorem:main} implies that $\Aut(X)$ has bounded finite
subgroups if and only if $A$ is a division algebra; see Remark~\ref{remark:SB-via-LAG} for details.
The following theorem, which we prove directly, amplifies this observation.

\begin{theorem}\label{theorem:SB}
Let $\KK$ be a field
that contains all roots of~$1$. Let~$X$ be a Severi--Brauer variety of dimension $n-1$ over~$\KK$,
and let $A$ be the corresponding central simple algebra. Assume that the characteristic
$\Char\KK$ of $\KK$ does not divide $n$.
The following assertions hold.
\begin{itemize}
\item[(i)] The group $\Aut(X)$ has bounded finite
subgroups if and only if $A$ is a division algebra; in particular, if $n$ is a prime number,
then $\Aut(X)$ has bounded finite
subgroups if and only if $X(\KK)=\varnothing$, i.e., $X$ is not isomorphic to~$\PP^{n-1}$.

\item[(ii)] Suppose that $A$ is a division algebra.  Let~$g\in\Aut(X)$ be an element of finite order,
and $\Gamma\subset\Aut(X)$ be a finite subgroup. Then $g^{n}=1$, and $\Gamma$ is an abelian group
whose order divides~$n^2$.
\end{itemize}
\end{theorem}

In particular, if $\KK$ is a perfect field that contains all roots of~$1$, then
Theorem~\ref{theorem:SB} applies to all Severi--Brauer varieties over~$\KK$;
indeed, in this case $\Char\KK$ cannot divide the dimension of a central division algebra over~$\KK$,
see Remark~\ref{remark:char-vs-dim-perfect} below. In the case of an arbitrary field whose characteristic $p$ divides
the dimension of the corresponding division algebra, the structure of finite subgroups of the
automorphism group is still rather simple, see Proposition~\ref{proposition:SBp};
in particular, finite subgroups of the automorphism groups are always abelian in
this case as well. However, for such Severi--Brauer varieties finite $p$-subgroups of the automorphism
group can have arbitrarily large order, see Example~\ref{example:BMR}.

Applying Theorem~\ref{theorem:main} to a projective orthogonal group, one can prove that the automorphism group of a smooth quadric $Q$ over a field $\KK$ of characteristic different from~$2$ that contains all roots of~$1$ has bounded finite subgroups if and only if~\mbox{$Q(\KK)=\varnothing$}. We
find explicit bounds for orders of finite automorphism groups
of quadrics over appropriate fields, thus generalizing the results of \cite[\S4]{BandmanZarhin2015a} and making them more precise.

\begin{theorem}
\label{theorem:quadricnew}
Let $\KK$ be a field  that contains all roots of $1$. Assume that $\Char \KK \ne 2$ or~$\KK$ is perfect.
Let~\mbox{$n\ge 3$} be an  integer, and
let $Q\subset\PP^{n-1}$ be a smooth quadric hypersurface over~$\KK$.
The following assertions hold.
\begin{itemize}
\item[(i)] The group $\Aut(Q)$ has bounded finite
subgroups if and only if~\mbox{$Q(\KK)=\varnothing$}.

\item[(ii)] If $n$ is odd  and~\mbox{$Q(\KK)=\varnothing$}, then
every finite subgroup of $\Aut(Q)$ is isomorphic to~\mbox{$(\ZZ/2\ZZ)^m$}, where $m\le n-1$.

\item[(iii)] If $n$ is even and~\mbox{$Q(\KK)=\varnothing$}, then every
non-trivial element of finite order in the group $\Aut(Q)$
has order $2$ or $4$, and the order of every finite subgroup of $\Aut(Q)$ divides~$8^{n-1}$.
\end{itemize}
\end{theorem}

According to Theorem~\ref{theorem:quadricnew}(ii), if $Q$ is a smooth odd-dimensional quadric
without points over a field of characteristic different from $2$ that contains all roots of~$1$, then every finite group faithfully acting
on $Q$ is abelian. If $Q$ has even dimension, this is not always the case, see Example~\ref{example:Pfister}.

Note that Theorem~\ref{theorem:quadricnew} fails
over a non-perfect field
of characteristic~$2$,
see Example~\ref{example:BMR}.
Note also that one cannot drop the assumption on the existence of  roots of~$1$ in Theorems~\ref{theorem:SB} and~\ref{theorem:quadricnew}.
Indeed, the conic over the field of real numbers defined by the equation~\mbox{$x^2+y^2+z^2=0$}
has automorphisms of arbitrary finite order.

\smallskip
The plan of our paper is as follows.
In~\S\ref{section:tori} we prove Theorem \ref{theorem:main} in the case when $G$ is a torus. The proof is based on the  Minkowski theorem on finite subgroups of $\GL_n(\ZZ)$ and elementary Galois theory.

In~\S\ref{section:lin-alg-groups} we study finite subgroups of linear algebraic groups
and prove Theorem~\ref{theorem:LAG}, which is a more precise version of Theorem~\ref{theorem:main}. The idea  of the proof is the following. According to
a result of Borel and Tits, for every connected anisotropic reductive group $G$ over a perfect field~$\KK$, every element~\mbox{$g\in G(\KK)$} is contained in $T(\KK)$, for some torus $T\subset G$. Using the results of~\S\ref{section:tori} we bound the order of $g$. On the other hand, choosing a faithful representation of $G$
we get an embedding~\mbox{$G(\KK) \subset \GL_N(\KK)$} for some positive integer $N$. This, together with a Burnside type result due to~\cite{HerzogPraeger} (see Theorem~\ref{theorem:Burnside} below), proves that $G(\KK)$ has bounded finite subgroups.

In~\S\ref{section:SB} we describe automorphism groups of Severi--Brauer varieties and prove
Theorem~\ref{theorem:SB}.
In~\S\ref{section:quadrics} we prove
Theorem~\ref{theorem:quadricnew}.

\smallskip
Throughout the paper by $\bar{\KK}$ we denote an algebraic closure
of a field~$\KK$, and by $\KK^{sep}$ we denote a separable closure of~$\KK$
(recall that $\KK^{sep}=\bar{\KK}$ provided that $\KK$ is perfect).
Given a variety $X$ defined over~$\KK$,
for an arbitrary field extension~\mbox{$K\supset\KK$} we denote
by~$X_{K}$
the corresponding scalar extensions to~$K$, and by $X(K)$ we denote the set of $K$-points of~$X$.
Abusing notation a bit, we write $\PP^n$ for a projective space over a field $\KK$,
and similarly write $\Gm$ and $\Ga$ for the multiplicative and additive groups, respectively.

\smallskip
We are grateful to
J.-L.\,Colliot-Th\'el\`ene,
S.\,Gorchinsky, A.\,Kuznetsov, D.\,Timashev, and B.\,Zavyalov for useful discussions.
Special thanks go to J.-P.\,Serre who suggested to reformulate the results of the earlier
draft of this paper in terms of multiplicative bounds.
Constantin Shramov was partially supported by the Russian Academic Excellence Project~\mbox{``5-100''},
by Young Russian Mathematics award, and by the Foundation for the
Advancement of Theoretical Physics and Mathematics ``BASIS''.
Vadim Vologodsky was partially supported
by the  Laboratory of Mirror Symmetry NRU~HSE, RF government grant,
ag.~\mbox{N\textsuperscript{\underline{o}}~$14.641.31.0001$}.

\section{Tori}
\label{section:tori}

In this section we  study elements of finite order
in algebraic tori.

The following result is a famous theorem of H.\,Minkowski, see~\cite[\S1]{Minkowski-1887}
or~\mbox{\cite[Theorem~1]{Serre-07}}.

\begin{theorem}
\label{theorem:Minkowski}
For any positive integer $n$, the group $\GL_n(\mathbb{Z})$ has bounded finite subgroups.
\end{theorem}

Theorem~\ref{theorem:Minkowski} tells us that
the maximal order $\Upsilon_A(n)$ of a finite subgroup
in~\mbox{$\GL_n(\mathbb{Z})$}
and the least common multiple $\Upsilon_M(n)$ of the orders of such subgroups
are well defined constants depending only on~$n$. For small $n$, one can compute the values
of $\Upsilon_A(n)$ and~$\Upsilon_M(n)$.
For instance, we have
$$
\Upsilon_A(1)= \Upsilon_M(1)=2,\quad \Upsilon_A(2)=12, \quad  \Upsilon_M(2)=24, \quad \Upsilon_A(3)=\Upsilon_M(3)=48,
$$
see e.g.~\mbox{\cite[\S1.1]{Serre-07}}
and~\mbox{\cite[\S1]{Tahara}}.
In particular, neither of the bounds given by $\Upsilon_A(2)$ and $\Upsilon_M(2)$
is strictly stronger than the other one.

The following is the main technical result of this section.

\begin{lemma}\label{lemma:tori-bounded-subgroups-general-case}
Let $n$ and $d$ be positive integers.
Let $\KK$ be a field such that the characteristic of~$\KK$
does not divide~$d$, and $\KK$
contains a primitive $d$-th root of~$1$.
Let~$T$ be an anisotropic $n$-dimensional algebraic
torus over $\KK$ such that $T(\KK)$ contains a point of order~$d$.
Then~\mbox{$d\le \Upsilon_A(n)$} and $d$ divides $ \Upsilon_M(n)$.
\end{lemma}
\begin{proof}
Let  $\check T = \mathrm{Hom} (\Gm, T_{\KK^{sep}})$ be the lattice of cocharacters of $T$.
Recall (see~\mbox{\cite[\S8.12]{Borel}}) that the functor $T \mapsto \check T$ induces an equivalence between the category of algebraic tori over $\KK$ and the category of free abelian groups of finite rank equipped with
an action of the Galois group $\Gal(\KK^{sep}/\KK)$ such that the image of
the homomorphism~\mbox{$\Gal(\KK^{sep}/\KK) \to \Aut(\check T)$} is finite. Denote this image by~$\Theta$.

The group of $d$-torsion elements of $T(\KK^{sep})$ is isomorphic,
as a Galois module, to~\mbox{$\check T \otimes \mu_d$}, where $\mu_d$  is the group of $d$-th roots of unity in
$\KK^{sep}$.  Since $\KK$ contains a primitive $d$-th root of $1$,
the Galois module $\mu_d$ is the trivial module $\mathbb{Z}/d\mathbb{Z}$, so that the Galois
module~\mbox{$\check T \otimes \mu_d$}  is isomorphic to $\check T/d\check T$. Hence, a point $x\in T(\KK)$
of order $d$ can be viewed as a
$\Gal(\KK^{sep} /\KK)$-invariant element $\bar v \in \check T/d\check T $ of order $d$ (so that $m\bar{v}\neq 0$
for~\mbox{$m<d$}).
Let~\mbox{$v\in  \check T$} be any preimage of $\bar v$ under the projection $ \check T \to  \check T/d\check T$, and let
$$
w= \sum_{\theta \in \Theta} \theta(v).
$$
Since  $\bar v$ is  $\Gal(\KK^{sep} /\KK)$-invariant, the image of
$w$ in $\check T/d\check T$ is equal to $|\Theta|\bar v$.
On the other hand, it is clear that $w$ is a $\Gal(\KK^{sep} /\KK)$-invariant element of $\check T$.
By the above mentioned equivalence of categories, $w$ gives rise to
a non-trivial homomorphism  $\Gm \to T$, provided that $w$ itself is non-zero.
Since $T$ is anisotropic, we conclude that $w=0$. Therefore,
$|\Theta|$ divides $d$, and the required assertion follows.
\end{proof}

\begin{remark}
J.-L.\,Colliot-Th\'el\`ene pointed out to us that the proof of Lemma~\ref{lemma:tori-bounded-subgroups-general-case}
can be reformulated in the following way.
The short exact sequence of $\Theta$-modules
$$ 0 \longrightarrow \check T \stackrel{d}{\longrightarrow}\check T \longrightarrow \check T/d\check T \longrightarrow 0$$
gives rise to the long exact sequence of cohomology groups
\begin{equation}\label{longex}
\ldots\to  H^0(\Theta, \check T) \to
H^0(\Theta, \check T/d\check T) \to  H^1(\Theta, \check T) \to \ldots
\end{equation}
Since $T$ is anisotropic, we have that
 $H^0(\Theta, \check T)=0$. Thus, the second map in~\eqref{longex} is injective.
On the other hand, the group~\mbox{$H^1(\Theta, \check T)$} is annihilated
by $|\Theta|$ (see for instance~\mbox{\cite[Proposition~IV.6.3]{CF67}}).
It follows that the group  $H^0(\Theta, \check T/d\check T)$  of $d$-torsion points of $T(\KK)$  is also  annihilated by~$|\Theta|$.
\end{remark}

Lemma~\ref{lemma:tori-bounded-subgroups-general-case}  implies the following result.

\begin{corollary}\label{corollary:tori}
Let $\KK$ be a field
that contains all roots of~$1$, and let $T$ be an anisotropic $n$-dimensional
torus over $\KK$. Let $g\in T(\KK)$ be an element of finite order $d$,
and let~\mbox{$\Gamma\subset T(\KK)$} be a finite subgroup.
Then $d$ is not divisible by $\Char\KK$, one has $d\le \Upsilon_A(n)$, and~$d$ divides~$\Upsilon_M(n)$.
Moreover, one has $|\Gamma|\le \Upsilon_A(n)^n$, and $|\Gamma|$ divides~$ \Upsilon_M(n)^n$.
\end{corollary}

\begin{proof}
Note that if $\Char\KK=p$ is positive, then $T(\KK)$
does not contain elements of order~$p$,
because $T_{\KK^{sep}}\cong\Gm^n$, and there are no such elements
in $(\KK^{sep})^*$.
Thus, if there is an element of finite order $d$ in $T(\KK)$, then $d$ is not divisible by~$p$,
so that $d\le \Upsilon_A(n)$ and $d$ divides $ \Upsilon_M(n)$ by Lemma~\ref{lemma:tori-bounded-subgroups-general-case}.
It remains to notice that every finite subgroup $\Gamma$ of $T(\KK)$ is an abelian
group generated by at most $n$ elements. Therefore, $|\Gamma|\le \Upsilon_A(n)^n$
and $|\Gamma|$ divides~$ \Upsilon_M(n)^n$.
\end{proof}

\begin{example}
\label{example:P1-without-2-pts}
Let $\KK$ be a field
that contains all roots of~$1$, and let $T$ be a one-dimensional torus over $\KK$ that is
different from $\Gm$.
Since $\Upsilon_A(2)=\Upsilon_M(2)=2$, we conclude from
Corollary~\ref{corollary:tori} that every non-trivial finite subgroup of $T(\KK)$ has order~$2$.
If $\Char\KK=2$, then $T(\KK)$ does not contain non-trivial finite subgroups at all, because
there are no elements of order~$2$ in~$(\KK^{sep})^*$.
\end{example}

In certain cases the bound provided by Corollary~\ref{corollary:tori} can be improved.

\begin{example}\label{example:L-over-K-torus}
Let $\KK$ be a field that contains all roots of $1$, and let $\KK\subset \LL$ be a Galois extension of degree $n$.
Consider the torus $T=R_{\LL/\KK}\Gm/\Gm$
over $\KK$, where $R_{\LL/\KK}$ denotes the Weil restriction of scalars, and the embedding $\Gm \hookrightarrow R_{\LL/\KK} \Gm$
comes by adjunction from the identity morphism $\Gm \to \Gm $. Let $\Gamma$ be a finite subgroup of $T(\KK)\cong \LL^*/\KK^*$.
We claim that $|\Gamma|$ divides $n$.
Indeed, since $\KK$ contains all roots of $1$, there is a well-defined homomorphism
$$
\zeta\colon (\LL^*/\KK^*)_{tors}\to \Hom\big(\Gal(\LL/\KK),\mu_\infty(\KK)\big)
$$
which sends a torsion element $\alpha\in \LL^*/\KK^*$ to
the map
$$
\gamma\mapsto \frac{\gamma(\tilde{\alpha})}{\tilde{\alpha}},
$$
where $\tilde{\alpha}$ is an arbitrary preimage of $\alpha$ in $\LL^*$.
It is easy to see
that $\zeta$ is injective. Since~$\Gamma$ is a subgroup of $(\LL^*/\KK^*)_{tors}$, this implies that
$|\Gamma|$ divides~$n$.
\end{example}

\section{Linear algebraic groups}
\label{section:lin-alg-groups}

In this section we study finite subgroups
of  linear algebraic groups and prove
Theorem~\ref{theorem:main}.

Recall that a linear algebraic group $G$ over a field $\KK$ is a smooth closed subgroup scheme
of $\GL_N$ over $\KK$. In particular, the group $G(\KK)$ of its $\KK$-points
has a faithful finite-dimensional representation in a $\KK$-vector space.
We refer the reader to~\cite{Borel} and~\cite{Springer}
for the basics of the theory of linear algebraic groups.

A connected semi-simple algebraic group $G$ is said to be \emph{simply connected} if every central isogeny~\mbox{$\tilde G \to G$}, where $\tilde G$ is
a connected semi-simple group, is  necessarily  an isomorphism. Recall that every connected semi-simple  group $G$ admits a \emph{universal cover} which is a pair consisting of a connected semi-simple
simply connected group $\tilde G$ and a central isogeny~\mbox{$\tilde G \to G$}. The group scheme theoretic kernel of the latter isogeny
is called the  \emph{algebraic fundamental group}
of $G$ and is denoted by $\pi_1(G)$. This is a finite group scheme
whose order~$|\pi_1(G)|$ equals the order of the topological fundamental group of the
connected semi-simple group over $\mathbb{C}$ constructed from the the root datum of $G_{\bar{\KK}}$.

Let $H$ be a quasi-simple algebraic group over an algebraically closed field
(that is, $H$ has no proper infinite normal closed subgroups).
One defines the set $\mathcal{T}(H)$ of \emph{torsion primes of $H$} to be the empty set  if $H$ has type
$\mathrm{A}_n$ or $\mathrm{C}_n$. If $H$ has type $\mathrm{B}_n$, $\mathrm{D}_n$, or $\mathrm{G}_2$,
we set $\mathcal{T}(H)=\{2\}$; if  $H$ has type $\mathrm{F}_4$, $\mathrm{E}_6$, or $\mathrm{E}_7$, we set
$\mathcal{T}(H)= \{2, 3\}$; if  $H$ has type  $\mathrm{E}_8$, we set $\mathcal{T}(H)= \{2, 3, 5\}$.
Given any connected semi-simple algebraic group $G$ over a field $\KK$
we say that a prime $p$  is a \emph{torsion prime of $G$} if  $p$ is a torsion prime for some quasi-simple direct factor of $\tilde G_{\bar{\KK}}$, where $\tilde G$
is the universal cover of $G$.

Similarly to the case of an algebraically closed field,
many properties of linear algebraic groups are determined
by their maximal tori. Note that in general a linear algebraic group
$G$ over a non-algebraically closed field $\KK$
may contain non-isomorphic maximal tori,
but their dimension still equals the dimension
of maximal tori in $G_{\bar{\KK}}$, see~\mbox{\cite[Theorem~13.3.6(i)]{Springer}}
and~\cite[Remark~13.3.7]{Springer}; this dimension is called the \emph{rank} of~$G$.

Recall that an element $g \in G(\KK)$  is called semi-simple if  its image  in $\GL_N(\KK)$ is diagonalizable over an algebraic closure $\bar{\KK}$  of $\KK$. The notion of a
semi-simple element is intrinsic, that is,
it does not depend on the choice of $N$ and an
embedding~\mbox{$G \hookrightarrow \GL_N(\KK)$},  see~\mbox{\cite[\S2.4]{Springer}}.
The main tool that will allow
us to apply the results of~\S\ref{section:tori}
is the following theorem.

\begin{theorem}[{see~\cite[Corollary~13.3.8(i)]{Springer}}]
\label{theorem:Springerold}
Let $G$ be a
connected linear algebraic group over a field $\KK$, and
let $g\in G(\KK)$ be a  semi-simple  element. Then
there exists a torus $T\subset G$ such that
$g$ is contained in~$T(\KK)$.
\end{theorem}

\begin{corollary}\label{cor:Springerold}
Let ~$G$ be a
connected linear algebraic group over a field $\KK$, and
let~\mbox{$g\in G(\KK)$} be a finite order  element whose order is not divisible by the characteristic of~$\KK$.  Then
there exists a torus $T\subset G$ such that
$g$ is contained in~$T(\KK)$.
\end{corollary}

For anisotropic reductive  groups over perfect  fields  and for simply connected  semi-simple  anisotropic  groups over arbitrary fields
whose characteristic is large enough, one has a stronger result.

\begin{theorem}[{see \cite[Corollary~3.8]{BorelTits}} and  {\cite[Corollary 2.6]{Tits86}}]
\label{theorem:Springer}
Let~$G$ be a
connected anisotropic reductive group over $\KK$.
Assume,  in addition,   that either $\KK$ is perfect, or $G$ is  semi-simple,  simply connected, and $\Char \KK = p>0$ is not a torsion prime for $G$.
Then, for every element $g\in G(\KK)$,
there exists a torus $T\subset G$ such that
$g$ is contained in~$T(\KK)$.
\end{theorem}

\begin{corollary}\label{cor:Springer}
Under the assumptions of Theorem~\ref{theorem:Springer} the order of every finite order element of~$G(\KK)$
is not divisible by the characteristic of $\KK$.
\end{corollary}
\begin{proof}
By Theorem~\ref{theorem:Springer}  it suffices to prove the assertion in the case when $G$ is a torus,
in which case it is given by Corollary~\ref{corollary:tori}.
\end{proof}

Note that over fields of positive characteristic non-reductive linear algebraic groups
may have unbounded finite subgroups. For instance, the $p$-torsion subgroup of $\Ga$ over an infinite
field of characteristic $p$
is an infinite-dimensional vector space over the field~$\mathbf{F}_p$ of~$p$ elements.
However, this example is in a certain sense the only source of unboundedness
for unipotent groups.

\begin{lemma}\label{lemma:unipotent-torsion}
Let $G$ be a unipotent group over a field $\KK$. If $\Char \KK=0$, then the group~$G(\KK)$ does not contain elements of finite order greater than $1$.
If $\Char \KK=p>0$, then~$G(\KK)$  is a $p$-primary torsion group.
\end{lemma}
\begin{proof}
Without  loss of generality we may assume that $\KK$ is algebraically closed.  Then~$G$ is isomorphic to a closed subgroup of the group $U_n$ of unipotent upper triangular matrices (see e.g.~\cite[Proposition~2.4.12]{Springer}). Thus, we may assume that $G=U_n$. The lemma follows since $U_n$ can be obtained as a consecutive
extension of groups isomorphic to~$\Ga$.
\end{proof}

We will need the following auxiliary fact about
orders of finite groups with given exponents proved in~\cite{HerzogPraeger}.

\begin{theorem}
\label{theorem:Burnside}
Let $n$ and $d$ be positive integers,
and let $\KK$ be a field.
Let~\mbox{$\Gamma\subset\GL_n(\KK)$} be a finite subgroup.  If $\Char\KK>0$, denote by~$|\Gamma|'$
the largest factor of $|\Gamma|$ which is not divisible by
$\Char\KK$; otherwise put~\mbox{$|\Gamma|'=|\Gamma|$}.
Suppose that
for every~\mbox{$g\in \Gamma$} such that the order of $g$
is not divisible by the characteristic of $\KK$,  one has $g^d=1$.
Then~$|\Gamma|'$ divides~$d^n$.
\end{theorem}

\begin{proof}
It is proved in \cite[Theorem~1]{HerzogPraeger} that under the above assumptions
one has~\mbox{$|\Gamma|'\le d^n$}. Applying this to the $q$-Sylow subgroups $\Gamma_q$
for all primes $q\neq \Char\KK$, we see that $|\Gamma_q|\le d_q^n$, where $d_q$ is the largest
power of $q$ dividing $d$. This immediately implies that $|\Gamma_q|$ divides $d_q^n$,
and hence $|\Gamma|'=\prod_q |\Gamma_q|$ divides the number $\prod_q d_q$, which
in turn divides~$d$.
\end{proof}

Now we state and prove a more precise version of Theorem~\ref{theorem:main}.

\begin{theorem}\label{theorem:LAG}
Let $r$ and $n$ be positive integers.
Let $\KK$ be a field that contains all roots of~$1$,
and let $G$ be an anisotropic linear algebraic group
over $\KK$ such that the number of $\KK$-points of the group of connected components of $G$ is at most~$r$ and the rank of $G$ is at most $n$.
Denote by $N(G)$  the minimal dimension of a faithful representation of the maximal reductive quotient of
the neutral component of~$G_{\bar{\KK}}$.
Let $\Gamma$ be a finite subgroup of~$G(\KK)$.
The following assertions hold.
\begin{itemize}
\item[(i)]
If $G$ is reductive and   $\KK$ is perfect, then  $|\Gamma|$ divides~$r\Upsilon_M(n)^{N(G)}$.

\item[(ii)]
Suppose that $G$ is an arbitrary linear algebraic group. If $\Char\KK>0$, denote by~$|\Gamma|'$
the largest factor of $|\Gamma|$ which is not divisible by
$\Char\KK$; otherwise put~\mbox{$|\Gamma|'=|\Gamma|$}. Then~$|\Gamma|'$ divides~$r \Upsilon_M(n)^{N(G)}$.

\item[(iii)] Assume that $G$  is  connected,  semi-simple, and  $\Char \KK=p>0$  is not a torsion prime for $G$.
Write~\mbox{$|\pi_1(G)| = l p^m$},  for some non-negative integers $m$ and $l$ such that $l$ is not divisible by~$p$.
Then $\Gamma$
is a semi-direct product $\Gamma=\Gamma_1\rtimes \Gamma_2$ of its normal subgroup $\Gamma_1$ whose order is not divisible by~$p$,
and an abelian $p$-group $\Gamma_2$ of exponent less than or equal to~$p^m$. Moreover,  $|\Gamma_1|$ divides~$r \Upsilon_M(n)^{N(G)}$.
\end{itemize}
\end{theorem}

\begin{proof}
Clearly, we may assume that $G$ is connected, so that in particular $r=1$. Also, we assume that the
rank of~$G$ equals~$n$.
Note that assertion~(i) follows from  assertion~(ii) and
Corollary~\ref{cor:Springer}.

Let $g\in G(\KK)$ be an element of finite order not divisible by $\Char\KK$.
Then, by Corollary~\ref{cor:Springerold}, the element
$g$ is contained in some  subtorus of $G$.
Thus,  it follows from Corollary~\ref{corollary:tori}
that
$$
g^{\Upsilon_M(n)} =1.
$$

Let~\mbox{$R_u(G_{\bar{\KK}})$}
be the unipotent radical of~$G_{\bar{\KK}}$. (Note that unless $\KK$ is perfect the group~\mbox{$R_u(G_{\bar{\KK}})$}  need not be defined over $\KK$.)
Then $G_{\bar{\KK}}/R_u(G_{\bar{\KK}})$ is a
reductive group over $\bar{\KK}$.
By assumption,
the group~\mbox{$(G_{\bar{\KK}}/R_u(G_{\bar{\KK}}))(\bar{\KK})$}
admits a faithful representation in an $N(G)$-dimensional vector space over $\bar{\KK}$:
$$
(G_{\bar{\KK}}/R_u(G_{\bar{\KK}}))(\bar{\KK}) \hookrightarrow \GL_{N(G)}(\bar{\KK}).
$$
Composing this embedding  with the projection  $G(\KK) \to (G_{\bar{\KK}}/R_u(G_{\bar{\KK}}))(\bar{\KK})$ we construct a homomorphism
$$
\phi\colon G(\KK) \to \GL_{N(G)}(\bar{\KK}),
$$
whose kernel is contained in  $R_u(G_{\bar{\KK}})(\bar{\KK})$.
By Lemma~\ref{lemma:unipotent-torsion}, every element of finite order in~\mbox{$R_u(G_{\bar{\KK}})(\bar{\KK})$} has order divisible by~$\Char\KK$.
This means that the image $\phi(\Gamma)$ of a finite subgroup~\mbox{$\Gamma\subset G(\KK)$} in  $\GL_{N(G)}(\bar{\KK})$
  has order divisible by the largest factor $|\Gamma|'$
of $|\Gamma|$ not divisible by~$\Char\KK$; in particular,
if $\Char\KK=0$, then $\Gamma$ projects isomorphically to~\mbox{$(G_{\bar{\KK}}/R_u(G_{\bar{\KK}}))(\bar{\KK})$}.
Theorem~\ref{theorem:Burnside} applied to $\phi(\Gamma)$ implies that $|\Gamma|'$ divides $\Upsilon_M(n)^{N(G)}$.
This proves assertion~(ii).

For the proof of assertion~(iii), observe that since the group scheme $\pi_1(G)$ is commutative, we have that $ \pi_1(G)\cong Z \times Z'$,  where $Z$ is a group scheme of order~$p^m$
and  $Z'$ is a group scheme whose order is not divisible by~$p$.
The central extensions
\[
\aligned
& Z' \to \tilde G \to \tilde G/Z',\\
&Z \to \tilde G /Z' \to G
\endaligned
\]
give rise to the exact sequences of groups
\[
\aligned
& Z'(\KK)   \rar{}  \tilde G (\KK) \rar{} (\tilde G / Z') (\KK) \rar{} H^1_{fl}(\Spec \KK, Z'),\\
& Z(\KK)   \rar{}  (\tilde G /Z')(\KK) \rar{} G(\KK) \rar{N} H^1_{fl}(\Spec \KK, Z),
\endaligned
\]
where the groups on the right stand for cohomology of $Z'$ and $Z$  regarded as  sheaves for the
fppf  topology on $\Spec \KK$, and $H^1_{fl}$
denotes the first cohomology group
for the fppf  topology (see, for example, \cite[\S\,III.4]{Milne1980}).
Set $\Gamma_1=\Gamma \cap  \ker N$.  By Corollary \ref{cor:Springer} the group  $\tilde G (\KK)$ has no elements of order $p$. Since the multiplication by $p$ is invertible in~$Z'$, the same is true for   $H^1_{fl}(\Spec \KK, Z')$. Hence $\Gamma_1$ has no elements of order $p$. Thus, by assertion~(ii) the order of~$\Gamma_1$ divides~$\Upsilon_M(n)^{N(G)}$.
On the other hand, by construction~$\Gamma_1$ is a normal subgroup of $\Gamma$, and $\Gamma/\Gamma_1$ is a subgroup of~\mbox{$H^1_{fl}(\Spec \KK, Z)$}. The latter is an abelian group annihilated by~$p^m$.  Hence, the same is true for $\Gamma/\Gamma_1$. Finally, since $\Gamma_1$ has no elements of order $p$,  a $p$-Sylow subgroup of $\Gamma$ projects  isomorphically to~$\Gamma/\Gamma_1$.
Thus, the group~$\Gamma$ is isomorphic to a semi-direct product of $\Gamma_1$ and~$\Gamma/\Gamma_1$.
\end{proof}

Recall that for every positive integer $n$  there exists a finite collection of  reductive group schemes
$\mathcal{G}_1,\ldots,\mathcal{G}_{r(n)}$ over $\Spec \mathbb{Z}$, such that  every connected reductive group of rank at most $n$  over
an algebraically closed field $\KK$ can be obtained from one of $\mathcal{G}_i$'s via the base change along the morphism $\Spec \KK \to \Spec \mathbb{Z}$
(see~\mbox{\cite[Theorem~XXV.1.1]{SGA3}}).
By~\mbox{\cite[Proposition~VI\textsubscript{B}.13.2]{SGA3}}, every such  $\mathcal{G}_i$ admits an embedding into a group~$\GL_{N_i}$ over $\Spec \mathbb{Z}$
for some positive integer~$N_i$.
Thus, there exists a number $N(n)$ with the following property: for every
algebraically closed field $\KK$ and every connected reductive group~$G$ of rank at most $n$
over $\KK$, the group $G(\KK)$ has a faithful representation of dimension at most~$N(n)$.

\begin{corollary}\label{corollary:LAG}
In the notation of assertions (i), (ii), and~(iii) of Theorem~\ref{theorem:LAG}, the orders $|\Gamma|$, $|\Gamma|'$, and
$|\Gamma_1|$, respectively, divide the number~$r\Upsilon_M(n)^{N(n)}$.
\end{corollary}

\begin{proof}
Let us use the notation of the proof of Theorem~\ref{theorem:LAG}.
The group~\mbox{$G_{\bar{\KK}}/R_u(G_{\bar{\KK}})$}
has the same rank as~$G_{\bar{\KK}}$
(which is equal to the rank of~$G$).
Indeed, the rank of a unipotent algebraic group  is zero. Hence, the rank of~\mbox{$G_{\bar{\KK}}/R_u(G_{\bar{\KK}})$} is greater than or equal to  the rank of~$G_{\bar{\KK}}$. On the other hand, by~\cite[Theorem~10.6(4)]{Borel},
every extension of a torus by
a connected unipotent group over $\bar{\KK}$ admits a section, which means that the rank  of~$G_{\bar{\KK}}$ is greater than or equal to the rank of the quotient~\mbox{$G_{\bar{\KK}}/R_u(G_{\bar{\KK}})$}.
On the other hand,~\mbox{$G_{\bar{\KK}}/R_u(G_{\bar{\KK}})$} is a connected reductive group.
Therefore, it admits a faithful representation in an $N(n)$-dimensional vector space over~$\bar{\KK}$.
\end{proof}

\section{Severi--Brauer varieties}
\label{section:SB}

In this section
we describe automorphism groups of Severi--Brauer varieties and prove Theorem~\ref{theorem:SB},
as well as some more special results for  Severi--Brauer varieties over non-perfect fields.
 We refer the reader to \cite{Artin1982} for the definition and basic
facts concerning Severi--Brauer varieties.

Let $A$ be a central simple algebra of dimension $n^2$
over an arbitrary field $\KK$, and let~$\mathcal{A}^*$ be the corresponding inner form of  the algebraic group $\GL_n$. By definition, for every  scheme~$S$ over $\KK$, the group $\mathcal{A}^*(S)$ is the group of invertible elements
in the algebra~\mbox{$A \otimes_\KK  \mathcal{O}(S)$}.  Also denote by $\mathcal{A}^*/\G_m$ the quotient of $\mathcal{A}^*$ by its center. The latter is an inner form of~$\PGL_n$.
Let $X$ be the
Severi--Brauer variety corresponding to $A$.  Recall that~$X$ represents the functor that takes a scheme $S$ over $\KK$ to the set of right ideals $I$  in the sheaf of algebras~\mbox{$A\otimes _{\KK} \mathcal{O}_{S}$} which are locally free of
rank~$n$ as $\mathcal{O}_{S}$-modules and such that the quotient~\mbox{$(A\otimes _{\KK} \mathcal{O}_{S})/I$} is also locally free.
The action of the group~\mbox{$\mathcal{A}^*(S)$} on~\mbox{$A \otimes_\KK  \mathcal{O}_{S}$} by conjugation induces a homomorphism
\begin{equation}\label{grschofautsb}
\mathcal{A}^*/\G_m \to \underline{\Aut}(X),
\end{equation}
where the target is the group scheme of automorphisms of $X$.

The following fact is well known to experts (cf. Theorem~E on page~266
of~\cite{Chatelet}, or~\cite[\S1.6.1]{Artin1982}),
but for the reader's
convenience we provide a proof.

\begin{lemma}\label{lemma:SB-Aut}
Homomorphism~\eqref{grschofautsb} is an isomorphism. Moreover, it induces an isomorphism $\Aut(X)\cong A^*/\KK^*$.
\end{lemma}
\begin{proof}
For the first assertion, it suffices to prove that~\eqref{grschofautsb} is an isomorphism after the base change to $\KK^{sep}$. But $A\otimes _{\KK} \KK^{sep}$ is the matrix algebra and
$X_{\KK^{sep}}\cong\PP_{\KK^{sep}}^{n-1}$. Thus, the base change~\eqref{grschofautsb} boils down to the natural homomorphism $\PGL_n \to \underline{\Aut}(\PP_{\KK^{sep}}^{n-1})$
which is known to be an isomorphism.

For the second assertion,
consider  the exact sequence of groups with $\Gal(\KK^{sep}/\KK)$-action
$$
1\to (\KK^{sep})^* \to \mathcal{A}^*(\KK^{sep}) \to \underline{\Aut}(X)(\KK^{sep}) \to 1
$$
and the corresponding  exact sequence of Galois cohomology
groups
$$
1\to \KK^*\to A^*\to \Aut(X)  \to H^1(\Gal(\KK^{sep}/\KK), (\KK^{sep})^*).
$$
The latter cohomology group vanishes by Hilbert's Theorem~90,
and the assertion of the lemma follows.
\end{proof}

Recall the reduced norm homomorphism:
$$
\Norm \colon A^* \to \KK^*.
$$
One has $\Norm(cx) = c^n \Norm(x)$, for every
$c\in \KK^*$ and $x\in  A^*$. Hence, $\Norm$ induces a homomorphism
\begin{equation}\label{norm}
 A^*/\KK^* \to  \KK^*/(\KK^*)^n,
\end{equation}
where $(\KK^*)^n \subset  \KK^*$ is the subgroup of $n$-th powers.

\begin{lemma}\label{lemma:SB-exponent}
Let $n$ be a positive integer, let $\KK$ be a field  that contains all
roots of $1$, and let~$A$ be a central
division algebra of dimension $n^2$ over $\KK$.
Then, for every finite subgroup~\mbox{$\Gamma \subset A^*/\KK^*$}, the restriction of the homomorphism~\eqref{norm} to $\Gamma$ is injective:
$$
\Gamma \hookrightarrow \KK^*/(\KK^*)^n.
$$
In particular, $\Gamma$ is abelian and, for every element~\mbox{$g \in A^*/\KK^*$} of finite order,
the order of~$g$ divides $n$.
\end{lemma}

\begin{proof}
Denote by $A^*_1$ the kernel of the reduced norm homomorphism $\Norm \colon A^* \to \KK^*$.  We have an exact sequence of groups
$$
1\to \mu_n(\KK) \to A^*_1 \to A^*/\KK^* \to \KK^*/(\KK^*)^n,
$$
where $\mu_n(\KK) \subset \KK^*$ is the subgroup of $n$-th roots of unity.
In particular,
every element~\mbox{$g \in A^*/\KK^*$} of finite order whose image in $\KK^*/(\KK^*)^n$ is $1$
lifts to an element~\mbox{$\tilde g  \in A^*_1 \subset A^*$}, which also has  finite order.
Thus it suffices to prove that every element~\mbox{$x \in A^*$} of finite order belongs to $\KK^* \subset A^*$.
Assume that $x^d=1$. Using that~$\KK$ contains all
roots of~$1$, we get
$$
0= x^d -1= \prod_{\epsilon \in \mu_d(\KK)}(x -\epsilon).
$$
Since $A$ has no zero divisors, we conclude that $x\in \mu_d(\KK)$ as desired.
\end{proof}

We need an auxiliary result about bilinear forms on finite abelian groups.

\begin{lemma}\label{lemma:Gamma-bilinear-form}
Let $\Gamma$ be a finite abelian group, and let $B\colon \Gamma\otimes_{\ZZ}\Gamma\to \QQ/\ZZ$
be a homomorphism such that $B(g,g)=0$ for every $g\in\Gamma$. Then there exists a subgroup
$\Lambda\subset\Gamma$ such that the restriction of $B$ to $\Lambda\otimes_{\ZZ}\Lambda$ is zero, and $|\Gamma|$ divides $|\Lambda|^2$.
\end{lemma}

\begin{proof}
It is enough to prove the assertion in the case when $\Gamma$ is an $\ell$-group for some prime number $\ell$. We will do this
by induction on the order of~$\Gamma$.

Choose an element $g$ of maximal possible order $\ell^r$ in $\Gamma$, and
let $\langle g\rangle\cong \ZZ/\ell^r\ZZ$ be the cyclic group generated by~$g$.
Set
$$
\Gamma'=\{g'\in\Gamma\mid B(g,g')=0\}.
$$
Then $\Gamma'$ is a subgroup of $\Gamma$, and
there is an injective homomorphism
$$
\Gamma/\Gamma'  \hookrightarrow \Hom(\langle g\rangle,\QQ/\ZZ)\cong \ZZ/\ell^r\ZZ.
$$
Thus
$$
|\Gamma|\le |\langle g\rangle|\cdot |\Gamma'|=\ell^r|\Gamma'|.
$$

Note that $\Gamma'$ contains the cyclic group
$\langle g\rangle$.
Since $g$ has maximal possible order in~$\Gamma$, we have
$\Gamma'\cong \langle g\rangle \times\Gamma''$ for some subgroup $\Gamma''\subset\Gamma'$.
By induction, the group $\Gamma''$ contains a subgroup $\Lambda''$
such that the restriction of $B$ to $\Lambda''\otimes_{\ZZ}\Lambda''$ is zero, and $|\Gamma''|\le |\Lambda''|^2$.
Set~\mbox{$\Lambda=\langle g\rangle\times \Lambda''$}. Then  the restriction of $B$ to $\Lambda\otimes_{\ZZ}\Lambda$ is zero,
and
$$
|\Lambda|^2=\ell^{2r}|\Lambda''|^2\ge
\ell^{2r}|\Gamma''|=\ell^r|\Gamma'|\ge |\Gamma|.
$$
Since all these numbers are powers of $\ell$, the assertion of the lemma follows.
\end{proof}

\begin{corollary}\label{corollary:Gamma-bilinear-form}
Let $n$ be a positive integer, let $\KK$ be a field of characteristic $p\ge 0$ that contains all
roots of $1$, and let~$A$ be a central
division algebra of dimension $n^2$ over $\KK$.
Let~\mbox{$\Gamma \subset A^*/\KK^*$} be a finite subgroup.
Then $\Gamma\cong\Gamma_1\times\Gamma_2$, where
$\Gamma_2$ is a $p$-group, while the order of $\Gamma_1$
is not divisible by~$p$ and divides~$n^2$.
\end{corollary}

\begin{proof}
According to Lemma~\ref{lemma:SB-exponent}, the group $\Gamma$ is abelian.
Thus, if $p>0$, then~\mbox{$\Gamma\cong\Gamma_1\times\Gamma_2$} for its $p$-Sylow subgroup $\Gamma_2$ and a subgroup
$\Gamma_1$ of order not divisible by~$p$. If $p=0$, we set~\mbox{$\Gamma_1=\Gamma$} and let $\Gamma_2$ be a trivial group.
Denote by~$\mu_\infty(\KK)$ the subgroup of roots of~$1$ in~$\KK$.

Let the (infinite) group $\tilde{\Gamma}\subset A^*$ be the preimage of $\Gamma$
under the projection~\mbox{$A^*\to A^*/\KK^*$}.
Consider the commutator pairing
$$
B\colon\Gamma\otimes_\ZZ\Gamma\to\mu_\infty(\KK), \quad (g,h)\mapsto \tilde{g}\tilde{h}\tilde{g}^{-1}\tilde{h}^{-1},
$$
where $\tilde{g}$ and $\tilde{h}$ are arbitrary preimages of $g$ and $h$ in $\tilde{\Gamma}$.
Since $\tilde{\Gamma}$ is a central extension of the abelian group $\Gamma$, the map $B$ is a well-defined homomorphism.
Choosing an embedding~\mbox{$\mu_\infty(\KK)\hookrightarrow\QQ/\ZZ$}, we may consider $B$ as a $\ZZ$-bilinear pairing
with values in~\mbox{$\QQ/\ZZ$}.

Apply Lemma~\ref{lemma:Gamma-bilinear-form} to the restriction of the pairing $B$ to $\Gamma_1\otimes_\ZZ\Gamma_1$.
We infer that there exists a subgroup $\Lambda\subset\Gamma_1$ such that $B$ is trivial on $\Lambda\otimes_\ZZ\Lambda$
and $|\Gamma_1|$ divides $|\Lambda|^2$.
In particular, the preimage $\tilde{\Lambda}$ of $\Lambda$ in $\tilde{\Gamma}$ is abelian.
Let $\LL$ be the subalgebra of $A$ generated by $\tilde{\Lambda}$. Then $\LL$ is commutative, and thus it is a field.
Furthermore, $\LL$
is contained in the composite of field extensions such that each of them is obtained by adjoining to $\KK$
some element $a\in A$ with $a^k\in\KK$ for some positive integer $k$ not divisible by $p$.
Since $\KK$ contains all roots of $1$, an extension of this form is a splitting field of the polynomial~\mbox{$x^k-a$};
thus, it is a Galois extension of $\KK$. Therefore, $\LL$ is a Galois extension of $\KK$ as well.
On the other hand, the number~\mbox{$|\Gal(\LL/\KK)|=[\LL:\KK]$} divides~$n$,
see for instance~\mbox{\cite[\S10.3]{Bourbaki}}.

Finally, recall from Example~\ref{example:L-over-K-torus} that $|\Lambda|$ divides~\mbox{$|\Gal(\LL/\KK)|$}. Therefore,~$|\Gamma_1|$ divides~$n^2$.
\end{proof}

The following is a more precise version of Theorem~\ref{theorem:SB}(ii).

\begin{proposition}
\label{proposition:SBp}
Let $\KK$ be a field of characteristic $p\ge 0$
that contains all roots of~$1$. Let~$A$ be a central division algebra over $\KK$ of dimension~$n^2$,
and let $X$ be the corresponding Severi--Brauer variety.
Write~\mbox{$n=n' p^m$} for some non-negative integers~$m$ and~$n'$ such that~$n'$ is not divisible by~$p$.
Then every finite subgroup~\mbox{$\Gamma\subset\Aut(X)$}
is a direct product~\mbox{$\Gamma=\Gamma_1\times \Gamma_2$} of an abelian group $\Gamma_1$ whose exponent divides $n'$ and whose order divides~${n'}^2$,
and an abelian $p$-group~$\Gamma_2$ whose exponent divides~$p^m$.
\end{proposition}
\begin{proof}
By Lemma~\ref{lemma:SB-Aut},
one has $\Gamma\subset A^*/\KK^*$. Therefore, the assertion about the product structure and about
the order of~$\Gamma_1$ follows
from Corollary~\ref{corollary:Gamma-bilinear-form}.
The assertion about the exponents of~$\Gamma_1$ and~$\Gamma_2$ follows from Lemma~\ref{lemma:SB-exponent}.
\end{proof}

Now we are ready to prove Theorem~\ref{theorem:SB}.

\begin{proof}[Proof of Theorem~\ref{theorem:SB}]
Assertion~(ii) is a particular case of Proposition~\ref{proposition:SBp}.
Also, Proposition~\ref{proposition:SBp} tells us that if $A$ is a central division algebra,
then the group $\Aut(X)$ has bounded finite subgroups.

Now suppose that $A$ is not a division algebra. Then
$$
A\cong D\otimes_{\KK}\mathrm{Mat}_m(\KK)
$$
for some $2\le m\le n$, where $\mathrm{Mat}_m(\KK)$ denotes the algebra of $m\times m$-matrices,
and $D$ is a central division algebra over~$\KK$, see for instance~\mbox{\cite[Theorem~2.1.3]{GilleSzamuely}}.
Thus~$A$ contains a subalgebra isomorphic to $\mathrm{Mat}_m(\KK)$. Since the field~$\KK$ contains
roots of $1$ of arbitrarily large degree, we see from Lemma~\ref{lemma:SB-Aut} that the group $\Aut(X)$ contains elements of
arbitrarily large finite order. This completes the proof of assertion~(i).
\end{proof}

\begin{remark}\label{remark:SB-via-LAG}
Let $A$ be a central simple algebra  over a field $\KK$.  Denote by $ \mathcal{A}^*$ the algebraic group whose  $S$-points are invertible elements in the algebra $A \otimes _\KK  \mathcal{O}_S$. We have a natural embedding
$\Gm \hookrightarrow \mathcal{A}^*$ induced by the homomorphism $\mathcal{O}_S^* \hookrightarrow (A \otimes _\KK  \mathcal{O}_S)^*$.  The quotient group scheme
$\mathcal{A}^*/\Gm$ is anisotropic if and only if $A$ is a division algebra (see, for instance,~\mbox{\cite[\S23.1]{Borel}}).
In particular,  if $\KK$ is perfect,  then Theorem~\ref{theorem:SB}(i) follows from Theorem~\ref{theorem:main} applied to the reductive group $\mathcal{A}^*/\Gm$.

Similarly,
under the assumptions of  Proposition~\ref{proposition:SBp},  the fact that every finite subgroup of  $A^*/\KK^*$ is a
\emph{semi}-direct
product of its abelian $p$-Sylow subgroup  and a normal subgroup of bounded order is a special case of
Theorem~\ref{theorem:LAG}(iii) applied to the reductive group~\mbox{$\mathcal{A}^*/\Gm$}.
Indeed, the algebraic fundamental group of  $\mathcal{A}^*/\Gm$ is isomorphic to the group scheme $\mu_n=\ker(\Gm\rar{n} \Gm)$,
which has order $n$. Also, since  $\mathcal{A}^*/\Gm$
has type~$\mathrm{A}_{n-1}$, the set of torsion primes for  $\mathcal{A}^*/\Gm$ is empty.
\end{remark}

\begin{remark}\label{remark:char-vs-dim-perfect}
Let $\KK$ be a perfect field of positive characteristic $p$, and let $A$ be a central division algebra of dimension~$n^2$ over $\KK$.
Then $p$ does not divide $n$. Indeed, the Frobenius morphism
$\mathrm{Fr}\colon \bar{\KK}^* \to \bar{\KK}^*$ is an isomorphism and, hence, the Brauer group
$$
\Br(\KK)\cong H^2(\Gal(\bar{\KK}/\KK), \bar{\KK}^*)
$$
of $\KK$ has no $p$-torsion elements and it is $p$-divisible. Therefore, our claim  follows from the fact that, over any field, the dimension of a central division algebra and the order of its class in the Brauer group have the same prime factors
(see for instance~\mbox{\cite[Lemma~2.1.1.3]{Lieblich}}).
\end{remark}

The restriction on the characteristic of $\KK$ in  Theorem~\ref{theorem:SB} is essential for validity of the statement.

\begin{example}
\label{example:BMR}
Let $F$ be a field of characteristic $p>0$, and
let~\mbox{$\KK=F(x, y)$} be the field of rational
functions in two variables, so that $\KK$ is a non-perfect field of characteristic~$p$.
Let $A$ be an algebra over~$\KK$
with generators $u$ and $v$
and relations
$$
v^p =x,\quad u^p=y,\quad vu-uv=1.
$$
Then $A$ is a central division algebra of dimension~$p^2$ over $\KK$. This is a special case of the Azumaya property of the ring of differential operators in characteristic~$p$ (see~\mbox{\cite[Theorem~2.2.3]{BMR08}}), but can be also checked directly.
Indeed, it suffices to check that $A_{\bar{\KK}}=A \otimes_{\KK} \bar{\KK}$
is the algebra of $p\times p$-matrices.
To see this, take the elements
$$
v'=v-x^{\frac{1}{p}},\quad   u'=u-y^{\frac{1}{p}}
$$
in  $A_{\bar{\KK}}$
with  $x^{\frac{1}{p}},  y^{\frac{1}{p}} \in \bar{\KK}$.
Then $v^{\prime p}= u^{\prime p}=0$ and $v'u'-u'v'=1$.
Define an action of~ $A_{\bar{\KK}}$ on
the $\bar{\KK}$-vector space $V=\bar{\KK}[z]/(z^{p})$ letting $u'$ act as the multiplication by~$z$
and~$v'$ as~$\frac{d}{dz}$. It is easy to verify that $V$ is an irreducible representation of
$A_{\bar{\KK}}$.  Hence, by the Jacobson density theorem
the homomorphism $A_{\bar{\KK}} \to \End_{\bar{\KK}}(V)$ is surjective.
Since the dimension
of~$A_{\bar{\KK}}$ is at most $p^2$ for obvious reasons,
the latter homomorphism
is actually an isomorphism.

Now, since the intersection of~\mbox{$F(v)^*$} with~$\KK^*$ is $F(x)^*=F(v^p)^*$,
we see that the group~\mbox{$A^*/\KK^*$}  contains~\mbox{$F(v)^*/F(v^p)^*$} as a subgroup. The latter
group is $p$-torsion, because the $p$-th power of any rational function in~$v$ is a rational function in~$v^p$.
In other words, it can be regarded as a
vector space over the field~$\mathbf{F}_p$ of~$p$ elements.
At the same time, $F(v)$ is a vector space of dimension~$p$ over $F(v^p)$,
so that~\mbox{$F(v)^*/F(v^p)^*$} can be thought of as a projective space of dimension~\mbox{$p-1>0$} over an infinite field~$F(v^p)$.
Thus, the set~\mbox{$F(v)^*/F(v^p)^*$} is infinite, which implies that~\mbox{$F(v)^*/F(v^p)^*$} is infinite-dimensional as
a vector space over~$\mathbf{F}_p$.
In particular, for~\mbox{$p=2$} this construction provides an example of a conic $C$ over a non-perfect field of characteristic~$2$
such that $C$ is acted on by elementary $2$-groups of arbitrarily large order.
\end{example}

\medskip
However, for central division algebras whose dimension is divisible by the characteristic of $\KK$
one can prove the following result.

\begin{proposition}\label{boundnessfornonperfect}
Let $A$ be  a central division algebra of dimension $n^2$ over a field $\KK$ of finite characteristic  $p$.
If there exists an element $v\in A$ that is inseparable over $\KK$
(i.e., the field extension~\mbox{$\KK(v)\supset \KK$} is inseparable), then
the group $A^*/\KK^*$ has unbounded finite subgroups.
On the other hand, if $\KK$ contains all roots of $1$, and
every element $v\in A$ is separable over $\KK$, then the group $A^*/\KK^*$ has bounded finite subgroups.
\end{proposition}

\begin{proof}
If $v\in A$ is not separable, then there exists a subfield
$$
\KK \subset \LL \subset \KK(v)
$$
which is a purely inseparable extension of $\KK$ of degree $p$. Then
every non-trivial element of the group~\mbox{$\LL^*/\KK^*$} has order~$p$. Since $|\LL^*/\KK^*|=\infty$,
we conclude that the group~\mbox{$\LL^*/\KK^* \subset A^*/\KK^*$} has unbounded finite subgroups.

Conversely, suppose that every element of $A$ is separable over $\KK$. Denote by $n'$  the largest factor of $n$ which is not divisible by $\Char\KK$. Then
for every element~\mbox{$v\in  A^*$} whose image in  $A^*/\KK^*$ has finite order, one has~\mbox{$v^{n'} \in \KK^*$}.
Indeed, otherwise~$v^{n'}$ would be inseparable over~$\KK$, because
$$
\left(v^{n'}\right)^{\frac{n}{n'}}=v^n\in\KK^*
$$
by Lemma~\ref{lemma:SB-exponent}.
Hence, the assertion follows from  Proposition~\ref{proposition:SBp}.
\end{proof}

\begin{remark}
If $p=2$ or $p=3$, then every  central
division algebra $A$ of dimension $p^2$ over a field $\KK$ of characteristic  $p$ contains an element $v\in A$ which is inseparable over $\KK$. For~\mbox{$p=3$}
this result was proved in  \cite{Wed21}.
For $p=2$ the assertion is easy. Indeed, choose a subfield~\mbox{$\LL \subset A$} of degree $2$ over $\KK$.
If the extension $\KK\subset\LL$ is inseparable, then we are done.
Suppose that it is separable.
Then  $\KK \subset \LL$ is a Galois extension.
Let~$\sigma$ be the generator of its Galois group.
By the Skolem--Noether theorem
(see e.g.~\mbox{\cite[\S10.1, Th\'eor\`eme 1]{Bourbaki}}) 
the action of~$\sigma$  on~$\LL$ extends to
an inner automorphism
of $A$ given by an element $v\in A^*$, that is,
$$
vuv^{-1}=\sigma(u)
$$
for every $u\in \LL$.
Since the degree of $\LL$ over $\KK$ is $2$, the automorphism $\sigma^2$ is trivial.
Hence~$v^2$ commutes with $\LL$. On the other hand, the centralizer
of $\LL$ in $A$ is $\LL$ itself, and so we have
$v^2\in \LL$. Observe that
$$
\sigma(v^2)=vv^2v^{-1}=v^2.
$$
Since~$\LL$ is a Galois extension of $\KK$, this means that the element $v^2$ is, in fact, contained in~$\KK$.
Thus the field extension $\KK \subset \KK(v)$ is not separable as desired.
In particular, using Proposition~\ref{boundnessfornonperfect} we see that, for $p=2$ or~\mbox{$p=3$} and a central
division algebra $A$ of dimension $p^2$ over a field $\KK$ of characteristic~$p$, the group $A^*/\KK^*$ has unbounded finite subgroups.
\end{remark}

\section{Quadrics}
\label{section:quadrics}

In this section we study automorphism groups of quadrics and
prove Theorem~\ref{theorem:quadricnew}.

Let $V$ be a finite-dimensional vector space over a field $\KK$,  and let $q$ be a non-degenerate quadratic form on $V$.
Recall that a quadratic form $q$ is said to be non-degenerate if the associated symmetric bilinear form
$$
B_q\colon V\times V \to \KK,\quad B_q(v,w)= q(v+w) -q(v) -q(w),
$$
is non-degenerate. If $\Char \KK =2$ the form $B_q$ is also alternating. Hence, in this case the dimension of $V$ must be even.

Denote by $\mathrm{O}(V,q)$ the orthogonal (linear algebraic) group corresponding to $q$.
The following result is well-known (see, for instance,~\cite[\S\S22.4, 22.6]{Borel}).

\begin{lemma}\label{lemma:O-anisotropic}
The group  $\mathrm{O}(V,q)$ is reductive. It is anisotropic if and only
if $q$ does not represent $0$.
\end{lemma}

We will also need the following structural result on quadratic forms over a perfect field of characteristic $2$ due to Arf (\cite{Arf41}).

\begin{lemma}\label{lemma:quadratic-form-in-char-2}
Let $V$ be a finite-dimensional vector space over a perfect field $\KK$ of characteristic $2$,  and let $q$ be a non-degenerate quadratic form on $V$ (so that in particular~\mbox{$\dim V=2k$} is even).  Then, for
 some coordinates $x_1, \ldots, x_{2k}$ on $V$, the quadratic form $q$ is given by
\begin{equation}\label{canforminchartwo}
q_a(x_1, \ldots , x_{2k})= x_1^2 +x_1 x_2 + a x_2^2 + x_3 x_4 +\ldots +x_{2k-1}x_{2k},
\end{equation}
where $a$ is an element of $\KK$.
Moreover, two quadratic forms $q_a(x_1, \ldots , x_{2k})$ and~\mbox{$q_{a'}(x_1, \ldots , x_{2k})$} are equivalent if and only if $a$ and $a'$ have the same image
in the cokernel of Artin--Schreier homomorphism
$$
\KK \to \KK,\quad c \mapsto c^2 -c,
$$
which is  the Arf invariant  of the quadratic form.
In particular, if $\dim V >2$ then every non-degenerate quadratic form on $V$ represents $0$.
\end{lemma}

\begin{lemma}[{cf. \cite[Lemma~2.1]{GA13}}]
\label{lemma:orderp-in-O}
Let $\KK$ be a field of characteristic $p>2$.
Let $V$ be a vector space over $\KK$, and let $q$ be a non-degenerate quadratic form on~$V$. Assume that~$q$ does not represent~$0$. Then  the group of  $\KK$-points of~\mbox{$\mathrm{O}(V,q)$} has no elements of order~$p$.
 \end{lemma}
\begin{proof}
Assuming the contrary, let $g\in  \mathrm{O}(V,q)(\KK)$ be an element of order $p$. Viewing  $g$  as a linear endomorphism of $V$, we have
$$
g^p -1 =(g-1)^p=0.
$$
Applying the Jordan Normal Form theorem to $g-1$,
we can find linearly independent vectors  $v_1, v_2 \in V$
such that $g (v_1) = v_1$  and $g (v_2) = v_1 + v_2$.  Thus, we have
$$
B_q(v_1, v_2) = B_q(g (v_1), g (v_2 )) = B_q(v_1, v_1) + B_q(v_1, v_2).
$$
Hence, we obtain $2q(v_1) = B_q(v_1, v_1)= 0$, that is, $q$ represents $0$.
\end{proof}

\begin{lemma}[{cf. \cite[Corollary~4.4]{BandmanZarhin2015a}}]
\label{lemma:subgroups-in-O}
Let $\KK$ be a  field that contains all roots of $1$.
Assume that $\Char \KK \ne 2$ or $\KK$ is perfect.
Suppose that~$q$ does not represent $0$.
Then every finite subgroup of~\mbox{$\mathrm{O}(V,q)(\KK)$}
is isomorphic to $(\ZZ/2\ZZ)^m$, where $m\le \dim V$.
\end{lemma}
\begin{proof}
By Lemma \ref{lemma:quadratic-form-in-char-2}, if $\KK$ is perfect and $\Char \KK=2$, then we must have $\dim V =2$.
In this case, by  Lemma \ref{lemma:O-anisotropic}, the group $\mathrm{O}(V,q)$ is anisotropic and, thus,
isomorphic to a semi-direct product of an anisotropic torus $T$ of rank $1$ and the finite group $\ZZ/2\ZZ$.
According to Example~\ref{example:P1-without-2-pts}, the group $T(\KK)$ does not contain non-trivial finite
subgroups, and hence every non-trivial finite subgroup of~$\mathrm{O}(V,q)$ is isomorphic to~$\ZZ/2\ZZ$.

Now assume  that $\Char\KK\neq 2$.
Let $g\in  \mathrm{O}(V,q)(\KK)$ be an element of finite order. By Lemma \ref{lemma:orderp-in-O} the order of $g$ is not divisible by the characteristic of $\KK$.
Since $\KK$ contains all roots of unity, it follows that every
such element~\mbox{$g\in  \mathrm{O}(V,q)(\KK)$} viewed as a linear endomorphism of $V$
is diagonalizable in an appropriate basis for $V$.
Moreover, since $q$ does not represent~$0$, the diagonal entries of the matrix
of $g$ in this basis must be equal to~\mbox{$\pm 1$}. Hence $g^2=1$.
It follows that every finite subgroup~\mbox{$\Gamma\subset\mathrm{O}(V,q)(\KK)$} is isomorphic to $(\ZZ/2\ZZ)^m$ for some
non-negative integer $m$. In particular, $\Gamma$ is abelian, and thus it is conjugate
to a subgroup of
the group of diagonal matrices in $\GL(V)$.
Hence, one has~\mbox{$m\le \dim V$}.
\end{proof}

Now we are ready to prove
Theorem~\ref{theorem:quadricnew}.

\begin{proof}[Proof of Theorem~\ref{theorem:quadricnew}]
Let $V$ be an $n$-dimensional vector space such that $\PP^{n-1}$ is identified
with the projectivization $\PP(V)$, and let $q$ be a quadratic form
corresponding to the quadric~$Q$.

First, assume that $\KK$ is a perfect field of characteristic $2$. If $n$ is even, then the quadratic form is non-degenerate;
indeed, otherwise its kernel $T$ would be at least two-dimensional, so that the singular locus of $Q$,
which is $Q\cap\PP(T)$, would be non-empty.
Thus, by Lemma~\ref{lemma:quadratic-form-in-char-2} one has~\mbox{$Q(\KK)\ne \varnothing$}.  Moreover, writing  $q$ in the form~\eqref{canforminchartwo},
we see that~\mbox{$\Aut(Q)$} contains a subgroup isomorphic to~$\KK^*$.
Hence, $\Aut(Q)$ has unbounded finite subgroups. If $n$ is odd, the symmetric bilinear form
$B_q\colon V\times V \to \KK$ associated to $q$ has
a one-dimensional kernel.
In this case $q$ can be written as
$$
q(x_1, \ldots , x_n)= x_1^2 + r(x_2, \ldots, x_n)
$$
for some coordinates $x_1, \ldots, x_{n}$ on $V$ and some
non-degenerate quadratic form $r$ in $n-1$ variables.  Applying  Lemma ~\ref{lemma:quadratic-form-in-char-2} to  $r$, we see that  ~\mbox{$Q(\KK)\ne \varnothing$}
and $\Aut(Q)$ has unbounded finite
subgroups. In fact, in this case $\Aut(Q)$ is isomorphic to  the group of linear transformations
of the quotient $\bar{V}$ of $V$ by $T$
which preserve the induced bilinear form $\bar{B}_q$
on~$\bar{V}$, i.e., to the symplectic group
$\mathrm{Sp}(\bar{V}, \bar{B}_q)(\KK)$; see~\cite[\S22.6]{Borel}.
We see that in the case when $\Char \KK = 2$, assertion~(i) holds, while the assumptions of assertions~(ii), (iii),
and~(iv) do not hold.

From now on we assume that $\Char\KK \ne 2$.
The group $\Aut(Q)$ is isomorphic to the group of $\KK$-points of
the group scheme quotient $G= \mathrm{O}(V,q)/\mu_2$, where
$\mu_2 \subset  \mathrm{O}(V,q)$ is the central subgroup of order~$2$.
(More geometrically,  $\Aut(Q)$ can be identified with the group $\mathrm{PO}(V,q)$ of
automorphisms of the projective space  $\PP(V)$ that preserve $q$ up
to a scalar multiple.)
By Lemma~\ref{lemma:O-anisotropic},  the group $G $ is anisotropic if and only if~\mbox{$Q(\KK)=\varnothing$}.
The connected component of identity $G^{\circ}\subset G$  is a connected semi-simple algebraic group.
Moreover,  the algebraic fundamental
group  of $G^\circ $  has order $2$ if  $n$ is odd and order $4$  if $n$ is even.
Also, note that no prime other than $2$ is a torsion prime  for $G^{\circ}$. Thus, assertion~(i) of the proposition follows from Theorem~\ref{theorem:main}
applied to $G^{\circ}$. (It also follows that the group~\mbox{$\Aut(Q)$} has no elements of
order~\mbox{$\Char\KK$};  cf.  Lemma~\ref{lemma:O-anisotropic}.)
The reader will see that the argument we give below for the remaining assertions of the proposition also furnishes a direct proof of~(i).

If $n$ is odd, then the embedding
$\mu_2 \hookrightarrow \mathrm{O}(V,q)$ splits,
so that  $\Aut(Q)$ is isomorphic to the subgroup~\mbox{$\mathrm{SO}(V,q)(\KK)\subset \mathrm{O}(V,q)(\KK) $} of the orthogonal group that consists of matrices whose determinant is equal to $1$.
Thus, assertion~(ii) follows from
Lemma~\ref{lemma:subgroups-in-O}.

Suppose that $n$ is even.
We know from Lemma~\ref{lemma:subgroups-in-O} that
every non-trivial element of~\mbox{$\mathrm{O}(V,q)(\KK)$}
of finite order has order $2$.
Hence, using the exact sequence
\begin{equation*}
0\to \mu_2 \to  \mathrm{O}(V,q)(\KK)\to \Aut(Q) \to \KK^*/(\KK^*)^2,
\end{equation*}
we see that every non-trivial element of $\Aut(Q)$ of finite order
has order $2$ or $4$.
Let~\mbox{$\Gamma \subset \Aut(Q)$} be a finite subgroup.
Consider the embedding
$$
\Aut(Q)\cong G (\KK) \hookrightarrow  G(\bar{\KK}) \cong  \mathrm{O}(V,q)(\bar{\KK})/\{\pm 1\},
$$
and let $\tilde \Gamma$ be the preimage of $\Gamma$ in  $\mathrm{O}(V,q)(\bar{\KK})$. The order of every element of~$\tilde \Gamma$
divides~$8$, and the same is true for the subgroup
$\hat{\Gamma}\subset \GL(V)(\bar{\KK})$ generated by $\tilde \Gamma$
and scalar matrices whose orders divide $8$.
Thus, we conclude from Theorem \ref{theorem:Burnside} that
$|\hat{\Gamma}|$ divides~$8^n$.
On the other hand, we have~\mbox{$|\hat{\Gamma}|=8|\Gamma|$}. This proves
assertion~(iii).
\end{proof}

The next example shows that there exist even-dimensional quadrics satisfying the assumptions of
Theorem~\ref{theorem:quadricnew} and having a faithful action of non-abelian
finite groups.

\begin{example}\label{example:Pfister}
Choose an integer $k\ge 3$.
Let $\Bbbk$ be an algebraically closed field of characteristic zero, and let
$a_1,\ldots,a_k$ be independent transcendental variables.
Set~\mbox{$\KK=\Bbbk(a_1,\ldots,a_k)$}. Consider the Pfister
quadratic form
$$
q_k=\sum_I a_I x_I^2
$$
in the variables $x_I$, $I\subset \{1,\ldots,k\}$,
where $a_I=\prod_{i\in I} a_i$.
Let $Q$ be the quadric defined by the equation $q_k=0$ in
the projective space $\PP^{2^k-1}$ with homogeneous coordinates~$x_I$.

We claim that~\mbox{$Q(\KK)=\varnothing$}.
Indeed, if $P$ is a $\KK$-point on $Q$, its homogeneous coordinates can be written
as polynomials $p_I$ in $a_1,\ldots,a_k$, such that at least one of them is a non-zero polynomial.
Let $M$ denote the maximal degree of the polynomials $p_I$ in the variable $a_k$.
Write
$$
p_I=c_{I,M}a_k^M+c_{I, M-1}a_k^{M-1}+\ldots+c_{I,0},
$$
where $c_{I,j}$ are polynomials in $a_1,\ldots, a_{k-1}$; at least one of $c_{I,M}$ is a non-zero polynomial.
The fact that the quadratic form $q_k$ vanishes at $P$
implies that
$$
a_k^{2M}\sum\limits_{k\not\in I} a_Ic_{I,M}^2=0, \quad a_k^{2M+1}\sum\limits_{k\not\in I} a_I c_{I\cup\{k\},M}^2=a_k^{2M}\sum\limits_{k\in I} a_Ic_{I,M}^2=0.
$$
Thus, at least one of the $2^{k-1}$-tuples
$$
\big(c_{I,M}\big),\  k\not\in I, \quad\quad
\big(c_{I\cup\{k\},M}\big),\  k\not\in I,
$$
provides a point of the projective space $\PP^{2^{k-1}-1}$ over the field $\Bbbk(a_1,\ldots,a_{k-1})$ where the
quadratic form
$$
q_{k-1}=\sum_I a_I x_I^2, \quad I\subset \{1,\ldots,k-1\},
$$
vanishes. Proceeding by induction, we arrive to the conclusion that the quadratic form
$$
q_1(x_0,x_1)=x_{\varnothing}^2+a_{\{1\}}x_{\{1\}}^2
$$
represents zero over the field $\Bbbk(a_1)$, which is absurd.

Now consider the elements $\sigma, \tau\in\PGL_{2^k}(\KK)$ acting as
$$
\sigma(x_{\{1\}})=-x_{\{1\}}, \quad \sigma(x_{\{2\}})=-x_{\{2\}}, \quad \sigma(x_I)=x_I \text{\ for\ } I\neq \{1\}, \{2\},
$$
and
$$
\tau(x_I)=a_{\overline{I}}x_{\overline{I}},
$$
where $\overline{I}=\{1,\ldots,k\}\setminus I$.
Then both $\sigma$ and $\tau$ preserve the quadric $Q$, and~\mbox{$\sigma^2=\tau^2=1$}.
The commutator
$\iota=\sigma\circ\tau\circ\sigma\circ\tau$
acts as
\[
\aligned
\iota(x_I)=-x_I& \text{\ for\ }  I=\{1\}, \{2\}, \overline{\{1\}}, \overline{\{2\}},\\
\iota(x_I)=x_I& \text{\ for\ }  I\neq\{1\}, \{2\}, \overline{\{1\}}, \overline{\{2\}}.
\endaligned
\]
In particular, $\iota$ is a non-trivial element of $\PGL_{2^k}(\KK)$, i.e. $\sigma$ and $\tau$ do not
commute with each other.
On the other hand, it is clear that $\sigma$ and $\tau$ generate a finite subgroup in~\mbox{$\Aut(Q)$}.

Note that the algebraic group $\Aut(Q)$ has two connected components (both over $\KK$ and over $\bar{\KK}$),
corresponding to the connected components of the orthogonal group~\mbox{$\mathrm{O}(q_k)$}.
The element $\sigma\in \Aut(Q)$ clearly lifts to an element in the neutral component
$\mathrm{SO}(q_k)$ of~\mbox{$\mathrm{O}(q_k)$}, while $\tau$ lifts to an element of $\mathrm{SO}(q_k)_{\bar{\KK}}$.
Thus, both $\sigma$ and $\tau$ are contained in the neutral component of $\Aut(Q)$.
This gives an example of a finite non-abelian subgroup of a connected anisotropic
reductive group over a field containing all roots of unity, cf.~Lemma~\ref{lemma:SB-exponent}.
\end{example}

\end{document}